\documentclass[a4paper,12pt,draft]{amsart}
\usepackage{amsmath}
\usepackage{amsfonts}
\usepackage{amssymb,amsthm,amscd}
\usepackage[OT1]{fontenc}
\newtheorem{theorem}{Theorem}
\newtheorem{lemma}[theorem]{Lemma}
\newtheorem{corollary}[theorem]{Corollary}

\newtheorem{proposition}[theorem]{Proposition}
\newtheorem{definition}[theorem]{Definition}
\newtheorem{example}[theorem]{Example}

\newtheorem{Remark}[theorem]{Remark}
\numberwithin{theorem}{section}
\numberwithin{equation}{section}
\DeclareMathOperator \Holder {\hbox{\rm H\"older}}
\DeclareMathOperator \PSH {{\rm PSH}}
\DeclareMathOperator \MA {{\rm MA}}
\DeclareMathOperator \MAH {{\rm MAH}}
\DeclareMathOperator \Vol {{\rm Vol}}
\DeclareMathOperator \Id {{\rm Id}}
\DeclareMathOperator \Capa {{\rm Cap}}
\DeclareMathOperator \exph {{\rm exph}}
\DeclareMathOperator \diam {{\rm diam}}
\def\C{\mathbb{C}}
\def\R{\mathbb{R}}
\def\connection{D}
\def\om{\omega}
\def\e{\varepsilon}
\def\f{\varphi}
\def\dc{dd^c}
\def\p{\psi}
\def\d{\delta}
\begin{document}
\title{H\"older continuous solutions to Monge-Amp\`ere equations}
\author[J.-P.\ Demailly, S.\ Dinew, V.\ Guedj,  H.\ H.\ Pham,
 S.\ Ko\l odziej, A. Zeriahi]{Jean-Pierre Demailly, S\l awomir Dinew, Vincent Guedj, Pham Hoang Hiep, S\l awomir Ko\l odziej and Ahmed Zeriahi}

\date{\today}

\begin{abstract}
Let $(X,\omega)$ be a compact K\"ahler manifold. We obtain uniform H\"older regularity for solutions to the complex Monge-Amp\`ere equation on $X$ with $L^p$ right hand side, \hbox{$p>1$}. The same regularity is furthermore proved on the ample locus in any big cohomology class. We also study the range $\MAH(X,\omega)$ of the complex Monge-Amp\`ere operator acting on $\omega$-plurisubharmonic H\"older continuous functions. We show that this set is convex, by sharpening Ko\l odziej's result that measures with $L^p$-density belong to $\MAH(X,\omega)$ and proving that $\MAH(X,\omega)$ has the ``$L^p$-property'', $p>1$. We also describe accurately the symmetric measures it contains.

\end{abstract}

\maketitle

\section{introduction}
Let $(X,\omega)$ be a compact $n$-dimensional K\"ahler manifold. Let also  $\Theta = \{\theta\} \in H^ {1,1} (X,\R)$ be a given cohomology class on $X$. In the note we consider two different cases of interest:
\begin{enumerate}
\item $\Theta$ is a K\"ahler class, i.e. there exists a K\"ahler form which represents $\Theta$. In this case we assume without loss of generality that $\om\in\Theta$;
\item $\Theta$ is a {\it big} cohomology class, which means that there exists a (possibly singular) closed $(1,1)$ current $T$ representing $\Theta$  such that $T$ is \emph{strictly positive} i.e. $T\geq \e_0 \omega$ for some constant $\e_0>0$.
\end{enumerate}

The study of complex Monge-Amp\`ere equations on compact K\"ahler
manifolds with a K\"ahler background metric has a long history and
many spectacular results have appeared during the years. The big
cohomology class setting, on the other hand, was initiated
recently in \cite{BEGZ}. This is  the
most general setting where a meaningful (and nontrivial) theory
can be developed. Of course it covers the K\"ahler class setting
  as a particular case, but since the latter is
more classical and certain technicalities can be avoided we have
decided to treat both cases separately.

\smallskip

We deal with the K\"ahler setting first.
We study the range
of the (normalized) complex Monge-Amp\`ere operator
$$
\MA(u):=\frac{1}{V_{\omega}}(\omega+dd^c u)^n,
\; \; V_{\omega}:=Vol_{\omega}(X)=\int_X \omega^n,
$$
acting on $\omega$-plurisubharmonic ($\omega$-psh for short) H\"older-continuous functions $u$.
Here, as usual $d=\partial+\overline{\partial}$ and $d^c:=\frac{1}{2 i \pi}
(\partial-\overline{\partial})$, and $V_{\omega}$ denotes the volume of the cohomology class
$\{\omega\}$, so that $MA(u)$ is a probability measure.

This problem is motivated by the study of canonical metrics on mildly singular varieties: their
potentials
are solutions to degenerate complex Monge-Amp\`ere equations for which H\"older continuity is
the best global regularity one can expect. Furthermore even such weak regularity does imply
estimates for the ``metric'' $\omega+dd^c u$ which might be relevant for the study of the limiting
behavior of the K\"ahler-Ricci flow.
We refer the reader to \cite{ST1,ST2,ST3,SW,EGZ1,GKZ,K2,KT,BCHM,BEGZ,To,TZ,SW} for further
geometrical motivations and references.

We let $\PSH(X,\omega)$ denote the set of $\omega$-psh functions: these are defined as being locally equal to the sum of a plurisubharmonic and a smooth function and any such function $u$ additionally satisfies the inequality $\omega+dd^c u \geq 0$ in the weak sense of currents.

We let $\Holder(X,{\mathbb R})$ denote the set of real valued H\"older-continuous functions on $X$. Our goal is
thus  to understand the range
 $$
 \MAH(X,\omega):=\MA( \PSH(X,\omega) \cap \Holder(X,{\mathbb R})).
 $$

 A result of  fifth named author \rm \cite{K2} (see \cite{EGZ1} and \cite{Di1} for refinements in particular
 cases) asserts that a
 probability measure $\mu=f \omega^n$ which is absolutely continuous with respect to the
 Lebesgue measure belongs to $\MAH(X,\omega)$ if it has density $f \in L^p$
 for some $p>1$. Note that a Monge-Amp\`ere potential $u \in PSH(X,\omega)$
such that $MA(u)=\mu$ is unique, up to an additive constant.

The proof in [K3]   does not give any
 information on the H\"older exponent of the Monge-Amp\`ere potential.
We combine here the methods of \cite{K2} and the regularization techniques of
the first named author \cite{De82,De94} to establish the
 following result:

 \medskip

 \noindent {\bf Theorem A.}
 {\it
 Let $\mu=f \omega^n=MA(u)$ be a probability measure absolutely continuous with respect to Lebesgue
 measure,
 with density $f \in L^p$, $p>1$. Then $u$
 is H\"older-continuous with exponent $\alpha$ arbitrarily close to $2/(1+nq)$, where
 $q$ denotes the conjugate exponent of $p$.
 }

  \medskip

  It is a slightly better exponent than the one obtained in some special cases  in
  \cite{EGZ1} and \cite{Di1}. It is
 also asymptotically optimal (see \cite{Pl} and \cite{GKZ} for some local counterexamples which
 are easily adjustable to the
 compact setting). The proof
 uses a  subtle regularization result of \cite{De82,De94}, as in \cite{Di1} and \cite{BD}. The
extra  tool that allows us to remove symmetry/curvature constraints is the
 Kiselman minimum principle coupled with Demailly's method of attenuating singularities (the
 Kiselman-Legendre transform)
from \cite{De94}.

\smallskip

By keeping track of the H\"older constant together with the exponent one can in fact obtain
uniform estimates provided suitable control on the global geometry is assumed. More precisely if
we assume uniformly bounded geometry (this notion will be explained in the Preliminaries) the
following holds:

\medskip

\noindent{\bf Theorem A*.}
{\it Let $(X_s,\omega_s)$ be a family of compact K\"ahler manifolds with uniformly bounded geometry.
Consider the Monge-Amp\`ere equations $$(\omega_s+dd^c u_s)^n=f_s\omega_s^n,\ \ \sup_{X_s} u_s=0.$$
 If $\Vert f\Vert_{L^p(\omega_s^n)}\leq C$ are uniformly bounded then the solutions $u_s$ are uniformly
 H\"older continuous for any
 exponent $\alpha<2/(nq+1)$ and the H\"older constant is uniformly controlled by  $C$ and the
 constants from the definition of the
 uniformly bounded  geometry.
}
\medskip

We furthermore believe that additional technical improvements of our arguments may lead to
analogous statements in the case of classes which are merely semi-positive and {\it big}
(see \cite{BGZ} for a definition and further developments).

\smallskip

A satisfactory description of $\MAH(X,\omega)$ is yet to be found.
We nevertheless establish a technically involved
characterization (Theorem \ref{thm:char}) that
allows us to derive several useful consequences, for example we show:

   \medskip

 \noindent {\bf Theorem B.}
 {\it
 The set $\MAH(X,\omega)$ has the $L^p$-property: if $\mu \in \MAH(X,\omega)$
 and $0 \leq f \in L^p(\mu)$ with $p>1$ and $\int_X f d\mu=1$, then
 $f \mu \in \MAH(X,\omega)$.

 In particular the set $\MAH(X,\omega)$ is convex.
 }

  \medskip

  It has been recently proved by Dinh-Nguyen-Sibony \cite{DNS} (see also \cite{Ber} for recent
  developments) that
  measures in $\MAH(X,\omega)$ have the following strong integrability pro\-perty:
  if $\mu \in \MAH(X,\omega)$, then
  $$
  (\dag)
  \hskip1cm
  \exp( -\varepsilon \PSH(X,\omega)) \subset L^1(\mu),
  \text{ for some } \varepsilon>0.
  $$
  This is a useful generalization of Skoda's celebrated
  integrability theorem (see \cite{Sk,Ze}).

  It is natural to wonder whether condition $(\dag)$ characterizes $\MAH(X,\omega)$.
  This is the case when  $n=1$ (see \cite{DS} and Subsection \ref{onedim}). In this note we
 show that such a characterization still holds in higher dimensions provided the measures under consideration have symmetries:

   \medskip

 \noindent {\bf Theorem C.}
 {\it
 Let $\mu$ be a probability measure with finitely many isolated singularities of radial or toric
 type.
 Then $\mu$ belongs to $\MAH(X,\omega)$ if and only if $(\dag)$ is satisfied.
 }

  \medskip

Next we turn our attention to the general big cohomology class setting. To this end we choose a smooth $(1,1)$-form $\theta$ representing $\Theta$. Observe that in general one cannot have $\theta\geq 0$. Analogously to the K\"ahler setting we can nevertheless define   $\PSH(X,\theta)$ as the set of functions which are defined again as being locally equal to sum of a plurisubharmonic and a smooth function and any such function $\varphi$ should satisfy
  $\theta+dd^c \varphi \geq 0$. Observe that by assumption such functions exist, although they all may be singular.

It follows from the  regularization \rm theorem of the first
author \cite{De94}  that we can find a strictly positive closed
$(1,1)$ current $T_+=\theta+dd^c \f_+$ which represents $\Theta$
and has \emph{analytic singularities}, that is there exists $c>0$
such that locally on $X$ we have
$$
\f_+=c\log\sum_{j=1}^N|f_j|^2\text{ mod }C^\infty
$$
where $f_1,...,f_N$ are local holomorphic functions. Such a current $T_+$ is then smooth on a Zariski open subset $\Omega$, and the \emph{ample locus}  $\text{Amp}(\Theta)$ of $\Theta$ is defined as the largest such Zariski open subset (which exists by the Noetherian property of closed analytic subsets). Therefore any $\theta$-psh function $\psi$ with {\it minimal singularities} is locally bounded on the ample locus. Here having minimal singularity means that given any other $\theta$-psh function $\varphi$ one has the inequality
$$\varphi\leq\psi+O(1).$$

According to \cite{BEGZ} we can then define the (non-pluripolar) product
$\langle(\theta+\dc \f)^n\rangle$, and in case $\f$ has minimal singularities,
the total mass of this measure is independent of $\f$ and equals
$$
\int_X\langle(\theta+\dc \f)^n\rangle=:Vol(\Theta)>0.
$$
It is therefore meaningful to study  the (normalized) Monge-Amp\` ere equation
 $$
MA(\f):=\frac{1}{Vol(\Theta)}  (\theta + dd^c \f)^n = \mu,
 $$
for a given probability measure $\mu$ vanishing on pluripolar sets.

When $\mu=f dV$ is absolutely continuous with respect
 to Lebesgue measure with density  $f \in L^p (X)$, $p>1$, there is a unique solution modulo additive constant which turns out to have minimal singularities \cite{BEGZ}.
The solution is known to be globally continuous  on $X$ when the cohomology class $\Theta$ is moreover semi-positive (\cite{EGZ11}).

\smallskip

In this context we prove the following analogue of  Theorem A:

\medskip

\noindent {\bf Theorem D.}
 {\it  Let $\mu$ be a probability measure absolutely continuous with respect to a fixed smooth volume form, with density $f \in L^p (X)$, $p>1$. Let $\f$ be a weak solution of the Monge-Amp\`ere equation
$MA(\f)=\mu$. Then for any $0 < \alpha < 2/(1+nq)$, $\f$ is H\"older-continuous of exponent $\alpha$ locally in the ample locus $Amp({\Theta})$ of $\Theta$ (here
 $q$ denotes the conjugate exponent of $p$). }
 \medskip

  The note is organized as follows. In Section \ref{preliminaries} we recall all the basic facts
  and introduce the necessary
 definitions.
  Theorems A and A* are proved in  Section \ref{proof}. After recalling the one dimensional theory in
 Subsection \ref{onedim}, we establish the characterization of
  $\MAH(X,\omega)$ in Subsection \ref{characterization}. This allows us to prove Theorem B (in  Subsection \ref{thmB}) and derive further interesting consequences. The case of measures with symmetries is handled in
   Section \ref{symmetries}. Theorem $D$ is proved in Section \ref{big}. In the Appendix we briefly explain
 how bounds on the curvature allow to control the differential of the exponential mapping, a technical information needed in the proof of Theorem A*.

\section{preliminaries}\label{preliminaries}
\subsection{Curvature and regularization}
 Let $X$ be a compact K\"ahler manifold equipped with a
 fundamental K\"ahler \rm form $\omega $ given in local coordinates by
$$
\omega =\frac{i}{2}\sum_{k,j}g_{k\bar j } dz^k \wedge d\bar z ^j .
 $$
 Its
bisectional curvature tensor in those local coordinates is defined by
$$
R_{i\bar j k\bar l}:=-\frac{\partial^2g_{k\bar l}}{\partial z_i\partial \bar {z}_j}+\sum_{p,q=1}^ng^{p\bar{q}}
\frac{\partial g_{p\bar l}}{\partial \bar{z}_j}\frac{\partial g_{k\bar q}}{\partial z_i},
$$
with $g^{p\bar{q}}$ denoting the inverse transposed matrix of $g_{p\bar{q}}$ i.e.,
$\sum_{q=1}^ng^{p\bar{q}}g_{s\bar{q}}=\delta_{ps}$.
It is a classical fact that in the K\"ahler case the bisectional curvature tensor coincides with the Levi-Civita curvature tensor.
We say that the bisectional curvature is bounded by $A>0$ if for any $z\in X$ and any vectors $Z,
Y\in T_{z}X,\ Z, Y\neq 0$ we have the inequality
$$
\Big|\sum_{i,j,k,l}^nR_{i\bar j k\bar l}(z)Z_i\bar{Z}_jY_k\bar{Y}_l\Big|\leq A\Vert Z\Vert_{\omega}^2\Vert Y\Vert_{\omega}^2.
$$
Analogously the bisectional curvature is bounded from below (resp.\ from above) by $A$ if
$$
\sum_{i,j,k,l}^nR_{i\bar j k\bar l}(z)Z_i\bar{Z}_jY_k\bar{Y}_l\geq A\Vert Z\Vert_{\omega}^2\Vert Y\Vert_{\omega}^2,\qquad\hbox{(resp.\ ${}\leq{}$)}
$$
respectively. It is easy to check that the existence of such bounds is independent of the
choice of local coordinates.

\medskip

Recall that if $u$ is a psh function in a domain in ${\mathbb C}^n$
then a convolution with a radial smoothing kernel preserves positivity
of $ dd^c u $. For non-flat metrics, this may not be the case any
longer.  However, an approximation technique due to the first author allows to
control ``the negative part'' of the smooth form. It is described in
detail in \cite{De82} and \cite{De94}. Here we shall briefly highlight its main features.

Consider the exponential mapping from the tangent space of a given point $z\in X$
$$
\exp_z: T_zX\ni\zeta\mapsto \exp_z(\zeta)\in X,
$$
which is defined by $\exp_z(\zeta)=\gamma(1)$ with $\gamma$ being the geodesic starting from $z$
with initial velocity $\gamma'(0)=\zeta$. Given any function $u\in L^1(X)$,
we define its $\delta$-regularization $\rho_\delta u$\ to be
\begin{equation}\label{phie}
\rho_\delta u(z)=\frac{1}{\delta ^{2n}}\int_{\zeta\in T_{z}X}
u(\exp_z(\zeta))\rho\Big(\frac{|\zeta|^2_{\omega }}{\delta ^2}\Big)\,dV_{\omega}(\zeta),\ \delta>0
\end{equation}
according to \cite{De82}. Here $\rho$ is a smoothing kernel, $|\zeta|^2_{\omega }$\ stands for $\sum_{i,j=1}^ng_{i\bar{j}}(z)\zeta_i\bar{\zeta}_j$, and
$dV_{\omega}(\zeta)$\ is the induced measure $\frac1{2^nn!}(dd^c|\zeta|^2_{\omega } )^n$.
This may be formally  extended as a function on $X\times{\mathbb C}$ by putting  $U(z,w):=\rho_\delta u(z)$  for $w\in \mathbb C,\
|w|=\delta$.
The introduction of the variable $w$ is convenient  for an application of Kiselman minimum
principle \cite{Ki1,Ki2} to that function. It should be noticed that in
\cite{De94} the riemannian exponential map ``$\exp$'' has been replaced by
a ``holomorphic counterpart'' $\exph$, which is defined as the holomorphic
part of the Taylor expansion of \hbox{$\zeta\mapsto\exp_z(\zeta)$} (the
reason is that the calculations then become somewhat simpler, especially
in the non K\"ahler case, but this is not technically necessary; 
thanks to a well known theorem of E.~Borel, such a formal expansion can always
be achieved by a smooth function~$\exph:TX\to X$). The function $\exph$ is
however not  uniquely defined, and this weakens the intrinsic character of the estimates. Therefore, we stick here to the more
usual riemannian $\exp$ function. The estimates obtained in
\cite{De82} show that all results
of \cite{De94} and \cite{BD} are still valid with the unmodified definition
of $\rho_\delta u$, at least when $(X,\omega)$ is K\"ahler.
By Lemma 8.2 of \cite{De82}, the exponential function
$$
\exp:TX\rightarrow X,\quad
TX\ni(z,\zeta)\mapsto \exp_z(\zeta)\in X,\ \zeta\in T_{z}X
$$
satisfies the following properties:
\begin{enumerate}
\item $\exp$ is a $\mathcal C^{\infty}$\ smooth mapping;
\item $\forall z\in X,\ \exp_z(0)=z\ {\rm and}\ d_{\zeta}\exp(0)=\Id_{T_{z}X}$;
\item $\forall z\in X$\ the map $\zeta\rightarrow \exp_z{\zeta}$\ has a third order Taylor expansion at $\zeta=0$ of the form
\begin{equation}\label{taylor expansion}
\left|\, \exp_z (\zeta )_m-z_m -\zeta _m - \frac{1}{2}\sum _{j,k,l} R_{j\bar k l\bar m}\big(\bar{z}_k+{\textstyle\frac{1}{3}}\bar{\zeta}_k\big) \zeta _j \zeta _l \, \right| 
\leq C(|\zeta |^2 (|z|+|\zeta |)^2 ),\ \ |\zeta |<r,
\end{equation}
for small enough $r>0$. The expansion is valid in holomorphic normal coordinates
with respect to the K\"ahler metric.
\end{enumerate}
It is convenient to select a particular smoothing kernel, namely
 $\rho: \mathbb R_{+}\rightarrow\mathbb R_{+}$ by setting
$$\rho(t)=\begin{cases}\frac {\eta}{(1-t)^2}\exp(\frac 1{t-1})&\ {\rm if}\ 0\leq t\leq 1,\\0&\
{\rm if}\ t>1\end{cases}$$
 with a suitable constant $\eta$, such that
\begin{equation}\label{total integral}
\int_{\mathbb C^n}\rho(\Vert z\Vert^2)\,dV(z)=1
\end{equation}
($dV$ denotes the Lebesgue measure in $\mathbb C^n$).

The crucial estimate of the Hessian of $U(z,w)$ given in
\cite{De94}, Proposition 3.8 (see also \cite{De82}, Proposition 8.5),
coupled with Kiselman's theorem provide a lemma stated in this form
 in \cite[Lemma 1.12]{BD}:

\begin{lemma}
 Fix any bounded $\omega$-psh function $u$ on a compact K\"ahler manifold $(X, \omega )$. Let $U(z,w)$ be
 its regularization as defined above. Define the Kiselman-Legendre transform with level $c$ by
\begin{equation}\label{kisleg}
u_{c,\delta}:=\inf _{0\leq t\leq \delta }\Big[U(z,t)+Kt^2-K\delta^2-c\log\Big(\frac t{\delta}\Big)\Big].
\end{equation}

Then for  some positive constant $K$  depending on the curvature, the function $U(z,t)+Kt^2$ is increasing in $t$ and
one has the following estimate for the complex Hessian:
\begin{equation}\label{hessest}
\omega+dd^c u_{c,\delta}\geq -(A\min\{ c,\lambda(z,\delta)\}+K\delta^2)\,\omega,
\end{equation}
where $A$ is a lower bound of the negative part of the bisectional curvature of $\omega$,  while
$$
\lambda(z, t):=\frac{\partial}{\partial\log t}(U(z,t)+Kt^2).
$$
\end{lemma}

\subsection{Jensen formula and uniformly bounded geometry}

The classical Jensen formula (see, for example \cite{BT1})
for a $\mathcal C^2$ function $u$ defined in a ball $B(z,2\delta )$ in ${\mathbb C}^n$ says that
\begin{equation}\label{jensen}
  (\check{u}_{\delta } - u) (z )
=\frac{2n}{\delta ^{2n } \sigma _{ 2n-1} }
\int _0 ^{\delta } r^{2n-1} \int _0 ^r t^{1-2n} \int _ {|\zeta
|\leq t } \Delta u(z+\zeta )\,dV(\zeta ) \, dt\, dr,
\end{equation}
where $\check{u}_{\delta }$ is the average of $u$ over $B(z, \delta )$ and
 $ \sigma _{2n-1}$ denotes the total surface measure of the unit sphere.
 Now, if $u$ is defined in a large set, then the integration of the above
 formula in $z$ provides an estimate of the integral of
 $\delta ^{-2}(\check{u}_{\delta } - u)$ in terms of the integral of the Laplacian of $u$. We need such
 an
estimate on
 compact K\"ahler manifolds which is uniform as long as the geometry
 of manifolds is bounded in a certain sense.

\begin{definition}\label{unif}
Consider a family $(X_s, \omega _s )$ of  compact K\"ahler manifolds.
We shall say that it has \it uniformly bounded geometry  if
\begin{itemize}
\item[1)] the diameter $\diam(X_s,\omega_s)$ is uniformly bounded,
\item[2)] their bisectional curvatures are uniformly bounded,
\item[3)] the  injectivity radius is uniformly bounded from below.
\end{itemize}
\end{definition}

By well-known estimates \cite{HK}, it then follows that the total volumes $\Vol_{\omega_s}(X_s):=\int_{X_s}\omega_s^n$ are uniformly bounded above and below by constants $C$ and $C^{-1}$ independent of $s$.

It turns out that such bounds are enough to ensure various interesting geometric and analytic bounds. Note in particular that they imply lower bounds on the Tian $\alpha$ invariants for the classes of $\omega_s$-psh functions  which does not depend on $s$ (see \cite{BEGZ}).
In potential applications $X$ will usually stay fixed, while the K\"ahler forms may vary. Note that all conditions are obviously satisfied if the forms $\omega_s$ are bounded in $\mathcal C^{\infty}$ topology and uniformly positive; this can be achieved by selecting appropriate representatives when the cohomology classes $[\omega_s]$ are given and contained in a fixed relatively compact region of the K\"ahler cone of $X$. Thus an interesting case to treat would be when the classes $[\omega_s]$ approach the boundary of this cone. Unfortunately this may in general lead to a blow-up of the curvature and for this reason our argument cannot be applied to study the limiting behavior. On the other hand the method works if the forms $\omega_s$ approximate a $\mathcal C^{1,1}$ form $\omega$ in a fixed cohomology class provided that the curvatures of $\omega_s$ stay bounded.

We can now state a lemma to be used in the next section.

\begin{lemma}\label{Jen}
 Assume  $(X_s,\omega_s)$ is a family of compact K\"ahler manifolds with uniformly bounded geometry.
Let $u_s$ be continuous $\omega_s$-psh functions
 normalized by  $\min_{X_s}u_s =1,\ \max_{X_s}u_s\leq B$ for some fixed constant $B$. If
 $\rho_{\delta} u_s$ is the regularization of
 $u_s$ defined as in (\ref{phie}) then for $\delta$ small enough we have
$$
\int_{X_s} \frac{\rho_{\delta} u_s-u_s}{\delta ^2 }\omega_s ^n \leq C_0 ,
$$
where $C_0$ only depends on $B$ and the constants involved in the uniform bounds on the geometry.
\end{lemma}

\begin{proof} Let us fix $s$ and omit it in the notation for simplicity. By definition
\begin{eqnarray*}
\rho_\delta u(z)&=&\int_{\zeta\in T_zX}u(\exp_z\zeta)
\rho\Big(\frac{|\zeta|^2_{\omega}}{\delta^2}\Big)
\frac{dV_{\omega}(\zeta)}{\delta^{2n}}
=\int_{x\in X} u(x) \rho\Big(\frac{|\log_z x|^2_{\omega}}{\delta ^2}\Big)\,
\frac{dV_{\omega}(\log_z x)}{\delta^{2n}}\\
&=& \int_{x\in X}u(x) K_{\delta}(z,x)
\end{eqnarray*}
where $x\mapsto \zeta=\log_z x$ is the inverse of $\zeta\mapsto x=\exp_z(\zeta)$. The map $(z,x)\mapsto (z,\log_z x)$ defines a diffeomorphism from a neighborhood of the diagonal in  $X\times X$ onto a neighborhood of the zero section of $TX$ by the implicit function theorem. Here 
$$
K_\delta(z,x)=
\frac{1}{\delta^{2n}}\rho\Big(\frac{|\log_z x|^2_{\omega}}{\delta ^2}\Big)\,
dV_{\omega}(\log_z x)
$$
is the semipositive $(n,n)$ form on $X\times X$ defined as the pull-back of $\rho(|\zeta|^2_{\omega}/\delta^2)\,dV_{\omega}(\zeta)/\delta^{2n}$ by $(z,x)\mapsto\zeta=\log_z x$; it can be viewed as a kernel with compact support in a neighborhood of the diagonal of~$X\times X$. By definition,
we have $\int_{x\in X} K_{\delta}(z,x) = 1$ (as is clear by taking $u\equiv 1$), thus
$$
u(z)=\int_{x\in X}u(z)K_{\delta}(z,x).
$$
Therefore
\begin{eqnarray*}
\int_X \big(\rho_{\delta}u(z)-u(z)\big)dV_\omega(z)&=&
\int_{(x,z)\in X\times X}\big(u(x)-u(z)\big)K_\delta(z,x)\wedge dV_\omega(z)\\
&=&\int_{(x,z)\in X\times X}u(x)\big(K_\delta(z,x)\wedge dV_\omega(z)-
K_\delta(x,z)\wedge dV_\omega(x)\big)
\end{eqnarray*}
thanks to a change of variable $(z,x)\mapsto (x,z)$.
In order to finish the proof we need the following lemma which establishes a pointwise bound for the kernel:

\begin{lemma} 
If $d_{\omega}(z,x)\leq \delta$, then 
$$
|\left(K_\delta(z,x)\wedge dV_\omega(z)-
K_\delta(x,z)\wedge dV_\omega(x)\right)|\leq C\delta^{2-2n}dV_{\omega}(z)\wedge dV_{\omega}(x), 
$$
for some uniform constant $C$ which only depends on the curvature of $\om$. 
If   $d_{\omega}(z,x)> \delta$, then $K_\delta(z,x)\wedge dV_\omega(z)=
K_\delta(x,z)\wedge dV_\omega(x)=0$.
\end{lemma}

\begin{proof} 
Given the symmetry of $|\log_z(x)|_\omega=|\log_x(z)|_\omega=d_\omega(z,x)$, it is enough to bound the $(2n,2n)$-form $dV_{\omega}(\log_z x)\wedge dV_\omega(z)- dV_{\omega}(\log_x z)\wedge dV_\omega(x)$. 
The last assertion follows  from the fact that $ \rho\Big(\frac{|\log_z x|^2_{\omega}}{\delta ^2}\Big)= \rho\Big(\frac{|\log_x z|^2_{\omega}}{\delta ^2}\Big)=0$ if $d_{\omega}(z,x)> \delta$.

We now establish the first part of the lemma. Set $\zeta=\log_zx$ (i.e.\
$x=\exp_z(\zeta)$) and $y=\exp_z(\frac{\zeta}{2})=\exp_z(\frac{1}{2}\log_z(x))$ (the mid-point of the geodesic joining $z$ and $x$). Observe that from the expansion (\ref{taylor expansion}) applied at $y$ (which is identified with zero in this system of normal coordinates) we have 
\begin{equation}\label{logequation}
\zeta_m=\log_z(x)_m=x_m-z_m-\frac12\sum_{j,k,l}R_{j\bar{k}l\bar{m}}\Big(\bar{z}_k+\frac13(\overline{ x}_k-\overline{z}_k)\Big)(x_j-z_j)(x_l-z_l)+O(\Vert z-x\Vert^4).
\end{equation}
Now (\ref{logequation}) yields
$$
d\zeta_m=d(\log_z x)_m=dx_m-dz_m+O(\Vert z-x\Vert^2)(dx,dz),
$$
with an $O(...)$ term depending only on the curvature. By the choice of the center $y$ we have\break $z_j=\frac 12(z_j-x_j)+O(\Vert z-x\Vert^2)$, where the $O(...)$ term again only depends on the curvature. Thus the expansion
$$
dV_\omega(\zeta)=\frac{\omega(z)^n}{n!}(\zeta)=\Big(1-\sum _{j,k,l}R_{j\bar{k}l\bar{l}}z_j\bar{z}_k+O(\Vert z\Vert^3)\Big)
\frac{i}{2}d\zeta_1\wedge d\overline\zeta_1\wedge\ldots\wedge
\frac{i}{2}d\zeta_n\wedge d\overline\zeta_n
$$
at any given point $z$ yields
$$
dV_{\omega}(\log_z x)=\bigwedge_{j=1}^n\frac{i}{2}(dx_j-dz_j)\wedge (d\overline{x}_j-d\overline{z}_j)+O(\Vert z-x\Vert^2).
$$
Thus, by taking the product with $dV_\omega(z)$, exchanging $x$ and $z$, and then subtracting and dividing by $\delta^{2n}$, we obtain the desired bound
$$
\frac{dV_{\omega}(\log_z x)\wedge dV_{\omega}(z)-dV_{\omega}(\log_x z)\wedge dV_{\omega}(x)}{\delta^{2n}}=
\frac{O(\Vert z-x\Vert^2)}{\delta^{2n}}\,dV_{\omega}(z)\wedge dV_{\omega}(x).
$$ 
The appendix implies that $O(...)$ depends only on global bounds for the 
geometry.
\end{proof}

We can now use Fubini's theorem and the  estimates on the kernel to obtain
\begin{eqnarray*}
\lefteqn{\!\!\!\!\! \!\!\!\!\! \!\!\!\!\!\!\!\!\!
\int_{(x,z)\in X\times X}u(x)\big(K_\delta(z,x)\wedge dV_\omega(z)-
K_\delta(x,z)\wedge dV_\omega(x)\big) } \\
&=&\int_{x\in X}\int_{z\in B(x,\delta)}u(x)\big(K_\delta(z,x)\wedge dV_\omega(z)-
K_\delta(x,z)\wedge dV_\omega(x)\big)\\
&\leq& \int_{x\in X}\int_{z\in B(x,\delta)}|u(x)|C\delta^{2-2n}dV_{\omega}(z)\wedge dV_{\omega}(x) \\
&\leq&  \int_{x\in X}BC\delta^2 dV_{\omega}(x)\leq C_0\delta^2,
\end{eqnarray*}
as claimed. 
\end{proof}

\subsection{The $\mathcal H(\alpha)$ condition and measures uniformly dominated by capacity}

A fundamental tool in the study of $\omega$-psh functions  is the relative
capacity modelled on the Bedford-Taylor relative capacity (\cite{BT2}).

\begin{definition} \label{relcap}
Let $(X,\omega)$ be a compact K\"ahler manifold.
Given a Borel subset $K$ of~$X$, we define
its relative capacity with respect to $\omega$ by
$$
\Capa_{\omega}(K):=\sup\Big\{\int_K(\omega+dd^c \rho)^n|\ \rho\in \PSH(X,\omega),\ 0\leqslant\rho\leqslant 1\Big\}.
$$
\end{definition}

The following classes have been considered in \cite{EGZ1}:

\begin{definition}
 Let $\mu$ be a probability measure on a compact K\"ahler manifold $(X, \omega)$.
We say that $\mu$ belongs to the class ${\bf {\mathcal H}(\alpha)}$, $\alpha>0$ $($alternatively, that
$\mu$ statisfies the $\mathcal H(\alpha)$ property$)$, if
there exists $C_{\alpha}>0$ such that for any compact $K\subset X$,
 $$
 \mu(K)  \leq C_{\alpha} \Capa_{\omega}(K)^{1+\alpha}
 $$
  If this holds for any $\alpha>0$, we say that $\mu$ satisfies ${\bf  {\mathcal H}(\infty)}$.
\end{definition}

 It was proved in \cite{K1,K2} that measures of the type $\mu=f\omega^n$ with a density $f$ in $L^p$ for some $p>1$ do
 satisfy ${\mathcal H}(\infty)$ (see also \cite{Ze}).
A slightly stronger notion was introduced in \cite{DZ}:

\begin{definition}
 We say that a probability measure $\mu$ is {\rm dominated by capacity} for $L^p$
functions if there exists constants $
\alpha>0$ and $\beta >0$, such that
for any compact $K\subset X$ and
non-negative $f\in L^p(\mu)$
with $p>1$, one has for some
constant $C$ independent of $K$
that
$$\mu(K)\leqslant C\cdot \Capa_{\omega}
(K)^{1+\alpha}\ {\rm and}\ \int_K f\mu\leqslant
C\cdot \Capa_{\omega}(K)^{1+\beta }.$$
\end{definition}

Both notions are variations on the condition (A) introduced by 
fith named author  in \cite{K1}. These conditions, which are
actually stronger than condition (A), ensure the existence of
bounded solutions $u$ to
$$
MA(u)=f\mu,
$$
as long as $\int_X f d\mu=1$.

Note that the condition $\mathcal H(\infty)$ is equivalent to domination by capacity for $L^{\infty}$ functions by a simple application of the H\"older inequality.

\subsection{Big cohomology classes}
 Let $X$ be a compact K\"ahler manifold of dimension $n$,  and $\Theta = \{\theta\} \in H^ {1,1} (X,\C) \cap H^ 2 (X,\R)$ a big cohomology class with a smooth representative $\theta$.

We introduce the extremal function $V_\theta$ defined by
\begin{equation}\label{equ:extrem}
V_\theta(x):=\sup\{\f(x) \, |\,  \f\in PSH(X,\theta),\sup_X\f\le 0\},
\end{equation}
where $PSH(X,\theta)$ is the set of all $\theta$-plurisubharmonic functions on $X$.
The function $V_\theta$ is a $\theta$-psh function with \emph{minimal singularities}.

Similarly to the K\"ahler case we define the relative capacity:

\begin{definition} \label{thetacap}
Let $X$ be a compact K\"ahler manifold. Given a Borel subset $K$ of $X$, we define
its relative capacity with respect to $\theta$ by
$$
\Capa_{\theta}(K):=\sup\Big\{\int_K(\theta+dd^c \rho)^n \, | \, \rho\in \PSH(X,\theta),\ V_\theta(x)-1\leqslant\rho\leqslant V_\theta(x)\Big\}.
$$
\end{definition}

Observe that contrary to the K\"ahler case competitors to maximize the right hand side have minimal singularities but are in general {\it unbounded}. The Monge-Amp\`ere measures in the definition
are only considered outside the polar locus $\{x \in X \, | \, V_{\theta}(x)=-\infty\}$. Observe that
the latter depends on the cohomology class $\{ \theta\}$ but not on the choice of its representative $\theta$.

\smallskip

Most definitions from the K\"ahler setting have their {\it big} counterparts, we refer the readers to \cite{BEGZ} for details and more background regarding big cohomology classes. In particular we can apply the same convolution procedure to any $\theta$-psh function, as well as the Kiselman-Legendre transform.

\smallskip

In order to prove Theorem D we shall need a stability estimate proved in \cite{GZ11}:

\begin{proposition} \label{proEGZ}
Assume that $\mu$ is a probability measure absolutely continous with respect to a smooth volume form
$d V$, $ d \mu = f d V$, where $f \in L^p(X)$ with $p > 1$.
Let $\f,\p$ be  $\theta$-plurisubharmonic functions such that
$MA(\f)= \mu$,  $- M_0 + V_{\theta} \leq \varphi\leq V_{\theta}$ and $\psi\leq V_\theta$ on $X$, for some positive constant $M_0 >0$.
 Then for any exponent $0 < \gamma < \frac{1}{n q + 1}$, there exists a constant $B_0 = B_0(p,\gamma,M_0) > 0$ such that
 $$
  \sup_X (\psi - \varphi)_+ \leq B_0 \Vert (\psi - \varphi)_+ \Vert_{L^1 (X)}^{\gamma}.
 $$
\end{proposition}

\section{Proof of Theorems A and A*}\label{proof}

\begin{proof}[Proof of Theorem A]
Fix $u \in PSH(X,\omega)$ such that $MA(u)=\mu$.
 Denote by $A-1=A'>0$ a bound for the curvature of $(X, \omega)$.
 By \cite{K1} $u$ is continuous, so assume that $\min _X  u =1$ and denote by $B := \max_X u$
 the maximum of $u$.
 Consider $\rho_{\delta }u$- the regularization of the $\omega$-psh function $u$ defined  in (\ref{phie}).

 Let  us set for $\delta > 0$ and $\alpha > 0$,
 \begin{equation}\label{2new}
 E(\delta,\alpha) := \{(\rho_{\delta }u-u)(z) >  \delta ^{\alpha}\}
 \end{equation}
 Let  $0 < \alpha _1 <\frac{2}{qn +1}$.  Choose  $\varepsilon >0, \ \alpha ,\ \alpha_0$ such that
 $$
 \alpha_1 < \alpha < \alpha_0 < 2 - \alpha _0 q (n+\varepsilon).
 $$
 Set $\theta := e^{-3AB}$. Recall (Lemma 2.1) that there exists a constant $K$ which only
depends on the curvature such that the functions $\rho_{\delta }u+K\delta^2$ are increasing in $\delta$.
Note that for $\delta $ small enough
$ \theta^{\alpha_1}  \delta^{\alpha_1} \geq \delta^{\alpha _0}+K\delta^2(1-\theta^2)$.
Altogether this implies that
 $E(\delta,\alpha_0) \supset E (\theta \delta ,\alpha_1)$.

\smallskip

  We want to show that $E (\theta \delta ,\alpha_1)$ is empty.
   Recall the definition of the Kiselman-Legendre transform at level $\delta^{\alpha}$ (see Lemma 2.1)
 $$
 U_{\delta }= \inf _{ t\in [0,\delta ]}(\rho_{t }u + Kt^2 -\delta ^{\alpha }\log\frac{t}{\delta } -K\delta ^2 ),
 $$
 where $K$ is chosen as in the formula (\ref{kisleg}). It follows from \cite{De82} that the same $K$
 can be chosen for a family of manifolds with uniformly bounded geometry.
  In what follows $\delta_0$ and $\delta_j, c_j, \ j=1,2,3;$ denote constants which are uniform if the geometry is uniformly
 bounded
 and $\Vert f\Vert_p$
 stays bounded.

 By Lemma 2.1
 $$
 \omega +dd^c U_{\delta } \geq -[(A-1)\delta ^{\alpha }+K\delta ^2 ]\omega > -A\delta ^{\alpha } \omega +  2\delta^{\alpha_0}
 \omega
 $$
 for $0 < \delta  < \delta_0$, where $\delta_0 > 0$ is small enough. Therefore
 $$u_{\delta }:=\frac{1}{1+A\delta ^{\alpha }  } U_{\delta }$$
 is  $\omega$-psh on $X$ and satisfies
 $$
 \omega + dd^c u_{\delta} \geq \delta^{\alpha_0} \omega,
 $$
provided $A\delta^{\alpha}<1$, which we can safely assume.
 From Lemma \ref{Jen} we have
 \begin{equation}\label{used}
 \int _X \vert \rho_{\delta }u -u\vert \omega ^n \leq c_1 \delta ^2 , \ \
 \end{equation}
 for $0 < \delta < \delta_0$.
 Therefore for $E_0 = E (\delta,\alpha_0) = \{(\rho_{\delta }u-u)(z) > \delta ^{\alpha_0}\}$ we have
 $$
 \int _{E_0} \omega ^n \leq c_1 \delta ^{2-\alpha_0 },
 $$
 and, by H\"older inequality,
 $$
 \int _{E_0} f\omega ^n \leq c_2 \delta ^{(2-\alpha_0 )/q}.
 $$
 Let us modify $f$ setting $g=0$ on $E_0$ and  $g=cf$ elsewhere, with $c$ such that total integrals of $f$ and
 $g$ are equal.  Solve for continuous $\omega$-psh function $v$ (comp. \cite{K3})
 $$
 (\omega +dd^c v)^n = g\omega ^n ,\ \ \ \max (u-v)=\max (v-u).
 $$
 Observe that $\Vert f - g \Vert_{L^1 (X)} = 2 \int_{E_0} f \omega^n \leq 2 c_2 \delta ^{(2-\alpha _0 )/q}$. Then by  \cite{DZ}
 there
 exists $c_3 $ ($c_3$
 depends additionally on $\varepsilon>0)$ such that
 \begin{equation}\label{3}
 \Vert u-v\Vert_{L^{\infty} } \leq c_3 \delta ^{\frac{2-\alpha_0 }{q(n+\varepsilon ) } }.
 \end{equation}

  We claim that there exist small enough constants $\delta_1>\delta_2>\delta_3>0$ such that for any $0<\delta<\delta_3$ there is
 a set inclusion
 \begin{equation}
 \label{eq:FI}
 E (\theta \delta,\alpha_1) \subset  \{u_{\delta} - v > \delta^{\alpha} \}  \subset E (\delta,\alpha_0).
 \end{equation}

 Indeed, take $z$ in $E(\theta \delta, \alpha_1)$.
 By Lemma 2.1, the function $\rho_tu + K t^2$ is increasing in $t \in [0,\delta]$.
Thus for $t\in [\theta \delta , \delta ]$,
 $$
 \rho_tu (z) -u(z)=\rho_tu (z)-\rho_{\theta\delta}u+\rho_{\theta\delta}u
-u(z)\geq K(\theta\delta)^2-Kt^2+(\theta \delta)^{\alpha _1} \geq (\theta \delta )^{\alpha _1}- K\delta ^2,
 $$
 and for $t< \theta \delta $, since $\theta = e^{- 3 AB}$, we have
 $$
 -\delta ^{\alpha }\log (t/\delta ) \geq 3AB\delta ^{\alpha }.
 $$
 Therefore
 $$
(U_{\delta } - u)(z) \geq \min ((e^{-3AB}\delta )^{\alpha _1}- K\delta ^2 , 3AB\delta ^{\alpha }) = 3 AB\delta ^{\alpha }
 $$
 for  $0 < \delta <\delta _1$, where $\delta_1 > 0$ is small enough (we can safely assume that $\delta_1<\delta_0$). Hence, by
 (\ref{3})
 $$
(U_{\delta } -v)(z) \geq 3AB\delta ^{\alpha } - c_3 \delta ^{\frac{2-\alpha _0 }{q(n+\varepsilon ) }} >  2AB\delta ^{\alpha },
 $$
 for  $\delta <\delta _2$, where $0 < \delta_2 < \delta_1$ is small enough.
 Observe that
 $$
 U_{\delta }- u_{\delta }\leq AB \delta ^{\alpha }.
 $$
 Since $A B \geq 1$, it follows that $u_{\delta}(z) - v (z) > AB \delta^{\alpha} > \delta^\alpha$ for $\delta <\delta _2$, which
 proves the first
 inclusion
 $E(\theta \delta, \alpha_1) \subset \{ u_{\delta} - v > \delta^{\alpha}\}$ in (\ref{eq:FI}).

 To prove the second inclusion, take $z \notin E (\delta,\alpha_0)$. Since, under our assumptions
 $$
 u_{\delta } < U_{\delta } \leq \rho_{\delta }u,
 $$
 we get, applying (\ref{3})
$$
(u_{\delta }-v)(z) \leq (\rho_{\delta }u -u)(z)+c_3 \delta ^{\frac{2-\alpha _0 }{q(n+\varepsilon ) }}
 \leq \delta ^{\alpha _0} + c_3 \delta ^{\frac{2-\alpha _0 }{q(n+\varepsilon ) }}  < \delta ^{\alpha },
$$
for $0 < \delta <\delta_3,$ where $0 < \delta_3 < \delta_2$ is small enough. This proves our second inclusion
$$
\{ u_{\delta} - v > \delta^{\alpha}\} \subset  E (\delta,\alpha_0)
$$
for $0 < \delta <\delta_3$ and completes the proof of (\ref{eq:FI}).

  Now we want to apply the comparison principle do deduce from (\ref{eq:FI}) that the set $ E(\theta \delta,\alpha_1)$ is empty
 for $\delta > 0$ small enough.
Let us fix $0 < \delta < \delta_3$ and recall that $E_0 =  E (\delta,\alpha_0)$. From (\ref{eq:FI}) and the comparison principle \cite{K3}, if follows that
$$
\int _{ \{ u_{\delta}> v + \delta^{\alpha}\}} (dd^c u_{\delta} +\omega )^n \leq \int _{ \{ u_{\delta}> v + \delta^{\alpha}\}} (dd^c v +\omega )^n \leq
\int _{E_0} (dd^c v +\omega )^n =\int _{E_0} g\omega ^n =0.
$$
 Since $u_{\delta}$ is $\omega$-psh and $(\omega + dd^c u_{\delta})^n \geq \delta^{n \alpha_0} \omega^n$, it follows that the
 volume of
 the set $ \{ u_{\delta}> v + \delta^{\alpha}\}$ is zero. Hence it is empty, since $u_{\delta}$ and $v$ are $\omega-$psh
 functions.
 Therefore from (\ref{eq:FI}), it follows  that the set $ E(\theta \delta,\alpha_1)$ is also empty. Setting $\eta = \theta
 \delta = \delta
 e^{-3AB}$, we obtain
 $$
\rho_{\eta}u - u \leq e^{3 \alpha_1 A B} \ \eta^{\alpha_1},
$$
  for $ 0 < \eta < \eta_0 = e^{- 3 \alpha A B} \delta_3$.

Note that the above inequality  means that locally ["words
permutation"]  the $\eta$- convolution of $u$ is no more than
$u$ plus some constant of order $\eta^{\alpha_1}$. Thus repeating
the local argument from \cite{GKZ} one obtains that the supremum
of $u$ in a coordinate ball of radius $\eta$ and center $z$ is
also controlled by $u(z)$ and a constant of order
$\eta^{\alpha_1}$.
  This proves that $u$ is H\"older continuous with exponent $\alpha_1$.
\end{proof}

Note that in the proof above we could choose the same $\delta _1 , \delta _2$ and $\delta_3$ for uniform $\alpha _j , c_j$. Thus,
following the lines of this proof, one can obtain an analogous result for families of manifolds with uniformly bounded geometry.

\begin{theorem}[Theorem A*]\label{cor}
 Let $(X_s,\omega_s)$ be a family of $n$-dimensional compact K\"ahler manifolds with uniformly bounded geometry. Consider the Monge-Amp\`ere equations
 $$
(\omega_s+dd^c u_s)^n=f_s\omega_s^n,\qquad \sup_{X_s} u_s=0,
$$
where $\int_{X_s} f_s \omega_s^n=\int_{X_s} \omega_s^n.$

If $\Vert f\Vert_{L^p(\omega_s^n)}\leq C$ are uniformly bounded then the
 solutions $u_s$ are
 uniformly H\"older continuous for any exponent $\alpha<2/(nq+1)$ and the H\"older constant is uniformly controlled by  $C$ and
 the
 constants from the definition of the uniformly bounded  geometry.
\end{theorem}

As a direct application of Theorem A* one has the following corollary:
\begin{corollary} Suppose $X$ is a compact K\"ahler manifold and $\omega$ is a $\mathcal C^{1,1}$ smooth closed positive form on $X$. Suppose moreover that $\omega$ can be approximated in $\mathcal C^{1,1}$- norm by smooth closed forms with curvatures bounded by a fixed constant. Let also $f$ be any nonnegative function such that $f\in L^p(\omega^n)$ and $\int_Xf\omega^n=\int_X\omega^n$. Then the Monge-Amp\`ere equation
 $$(\omega+dd^c
 u)^n=f\omega^n, sup_{X} u=0$$
has an $\alpha$-H\"older continuous solution $u$ for any $\alpha< 2/(nq+1)$, where $q$ is the conjugate to $p$.
\end{corollary}

Finally we remark that in \cite{DZ} the stability result holds not only for measures absolutely continuous with respect to the Lebesgue measure, but also for any measure dominated by capacity for $L^p$ functions. Observe that in the proof the sole place where we used the assumption that $\mu$ is a measure with density was the application of the Jensen formula in (\ref{used}). Therefore by repeating the above proof one can get the following generalization:

\begin{proposition}\label{rem:geneThmA}
Let $u\in \PSH(X,\omega)$ solve the equation $MA(u)=\mu$ for $\mu$ a probability measure on a compact K\"ahler manifold $(X,\omega)$. Assume that $\mu$ satisfies the following additional assumptions:

i) {$\mu$ satisfies ${\mathcal H}(\infty)$};

ii) $\Vert \rho_{\delta}\phi-\phi \Vert_{L^1(\mu)}=O(\delta^b)$ for some $b>0$.

\noindent Then $u$ is H\"older continuous with  the 
exponent depending only on $n$ and $b$.
\end{proposition}

Examples of such singular measures have been considered in \cite{Hi}.

\section{Some properties of $\MAH(X,\omega)$}
\subsection{The one dimensional case}\label{onedim}

In this section we recall for reader's convenience the classical one dimensional theory of H\"older continuous potentials. We refer to \cite{DS} for more details. It is worthwhile to recall that the problem on Riemann surfaces is linear and hence much easier: analogous statements in the case of planar domains are classical in potential theory.

\begin{proposition} \label{pro:holderdim1} Let $(X,\omega)$ be a compact Riemann surface.
Let also $\mu=\omega+dd^c \phi$ be a probability measure on $X$, where $\phi \in PSH(X,\omega)$ and $\mathbb B(a,r)$ be the ball (with respect to the metric
induced by $\omega$) centered at point
$a$ with radius $r$.
The following properties are equivalent:

i) the function  $\phi$ is H\"older continuous ;

ii)  there exists constants $\alpha, C>0$ such that $\mu(\mathbb B(a,r)) \leq C r^{\alpha}$,
for all $a \in X$ and $0<r<1$;

iii) there exists $\varepsilon >0$ such that
$
 \exp\left( -\varepsilon PSH(X,\omega \right) ) \subset L^1(\mu).
$

\end{proposition}

\begin{Remark} \label{rem:Laplace}
As the Laplacian is a linear operator, Proposition  \ref{pro:holderdim1} is actually a local result.
It further holds for  higher dimensional subharmonic functions. We let the reader check that if $u$ is
a subharmonic function in some domain $\Omega \subset \mathbb R^n$
which contains the origin, and $0<\alpha<1$, then the
following are equivalent:
\begin{enumerate}
\item $\sup_{\mathbb B(\delta)} u -u(0) \leq C_1 \delta^{\alpha}$, for some $C_1>0$ and $0<\delta<<1$;
\item $\frac{1}{vol(\mathbb B(\delta))} \int_{\mathbb B(\delta)} u(z) \, dV(z) -u(0) \leq C_1 \delta^{\alpha}$, where $C_1>0$, \, $0<\delta<<1$;
\item $\int_{\mathbb B(\delta)} \Delta u  \leq C_3 \delta^{\alpha+n-2}$, for some $C_3>0$ and $0<\delta<<1$.
\end{enumerate}
It classically follows from this observation that any subharmonic function is $\alpha$-H\"older continuous
(respectively ${\mathcal C}^{1,\alpha}$) outside a set of arbitrarily small
$(n-2+\alpha)$-Hausdorff (respectively $(n-1+\alpha)$-Hausdorff) content.
\end{Remark}

\subsection{Characterization of $MAH(X,\omega)$}\label{characterization}

Let $\Omega$ be a bounded domain in $\mathbb C^n$. Analogously to the formula (\ref{jensen})
for each $u\in PSH (\Omega )$ and $\delta >0$ we set
$$
\check{u}_{\delta }(z)= \frac 1 {v_{2n}\delta ^{2n}} \int_{\mathbb B_\delta} u(z + w) dV (w)
\; \; \text{ and } \; \;
u_\delta (z) =\sup_{w\in\mathbb B_\delta } u(z+w),
$$
for $z\in\Omega_\delta =\{z\in\Omega: d(z,\partial\Omega )>\delta\}$.
Here
$$
\mathbb B_\delta =\{z\in\mathbb C^n:\ \Vert z\Vert =(|z_1|^2+...+|z_n|^2)^{\frac 1 2}<\delta \}
$$
and $v_{2n}$ is the volume of the unit ball $\mathbb B_1$.

\begin{theorem}  \label{thm:char}
Let $(X,\omega)$ be  a compact K\"ahler manifold, $\mu$  a positive Borel measure on $X$ so that
$\mu(X)=\int_X \omega^n$. The following are equivalent:
\begin{itemize}
\item[\rm i)] There exists a H\"older continuous $\omega$-psh $\varphi$ such that $\mu=(\omega+dd^c \varphi)^n$.
\smallskip
\item[\rm ii)] For every $z\in X$, there exists a neighborhood $D$ of $z$ and a H\"older continuous psh $v$ on $D$ such that $\mu |_D \leq (dd^c v)^n$.
\smallskip
\item[\rm iii)] $\mu\in\mathcal H({\infty})$  and there exists $C, \alpha >0$ such that
$
\int_K [\check{u}_{\delta }- u] d\mu \leq C \int_{\bar D} \Delta u \ \delta^{\alpha},
$
for all $u\in{\PSH}\cap L^\infty (\Omega)$, $K\subset\subset D\subset\subset \Omega$, where $\Omega$ is a local chart.
\end{itemize}
\end{theorem}

A positive measure $\mu$ thus  belongs to $MAH(X,\omega)$ if and only if it is locally the
Monge-Amp\`ere measure of a H\"older-continuous psh function.

\begin{proof}
 The implication i)~$\Rightarrow$~ii) is immediate.
The implication iii)~$\Rightarrow$~i) was observed to hold in Proposition \ref{rem:geneThmA}.

We now consider the implication ii)~$\Rightarrow$~iii).
It is enough to prove the inequality
$$
\int_K [\check{u}_{\delta } -u] (dd^ v)^n \leq C\int_{\overline D}\Delta u\delta^\alpha,
$$
for all $u\in \PSH\cap L^\infty (\Omega)$, $K\subset\subset D\subset\subset\Omega$ and for any local chart $\Omega$.

We can assume without loss of generality that $K=\mathbb B_1$ is the unit ball in $\mathbb C^n$, $D=\mathbb B_2$ and $-2\leq v\leq -1$, $|v(z)-v(w)|\leq \Vert z-w\Vert^s$ for all $z,w\in\mathbb B_2$. This implies that $h(z):=\Vert z\Vert^2-4<v$ on $\mathbb B_1$, while $v<h$ on $\mathbb B_2\backslash\mathbb B_{r_0}$ for some $1<r_0<2$.

Replacing $v$ by $\max (v,h)$ we can assume that $v=h$ on $\mathbb B_2\backslash\mathbb B_{r_0}$.
Fix $\rho\in C_0^\infty(\mathbb C^n)$ such that $\rho\geq 0$, $\rho(z) = \rho(\Vert z\Vert )$, supp$\rho\subset\mathbb B_1$ and $\int_{\mathbb C^n}\rho (z)\,dV(z)=1$. Set
$$\hat v_\delta (z)=\int_{\mathbb B_1} v(z-\delta w) \rho (w)\,dV (w)=\frac 1 {\delta^{2n}}\int_{\mathbb B(z,\delta )} v(w) \rho\Big(\frac {z-w}{\delta}\Big)\,dV (w).$$
Observe that
$$
(1)\ \ \ \ \ \ \hat v_\delta (z)-v(z)=\int_{\mathbb B_1} [v(z-\delta w)-v(z)]\rho (w)\,dV (w)\leq \delta ^s
$$
and
$$
(2)\ \ \ \ \ \ \ \ \ \ \ \
\left|\frac {\partial ^2\hat v_\delta}{\partial z_j\partial\overline z_k} (z) \right|\leq \frac {C\Vert v\Vert_{L^\infty (\Omega )} }{\delta ^2}, \ \ \ \ \ (dd^c \hat v_\delta )^n\leq \frac {C\,dV} {\delta^{2n}}.
$$

Choose now $\phi\in C_0^\infty(\mathbb C^n)$ such that $0\leq\phi\leq 1$, $\phi=1$ on $\mathbb B_{r_1}$ and supp$\phi\subset\mathbb B_{r_2}$, where $r_0<r_1<r_2<2$. Set
$$\overline v_{\delta}(z)=\int_{\mathbb B_1} v\big(z-\delta\phi (z) w\big) \rho (w)\,dV (w).$$
Observe that
$$(3)\ \ \ \ \ \ \overline v_\delta (z)-v(z)=\int_{\mathbb B_1} [v\big( z-\delta\phi(z) w\big)-v(z)]\rho (w)\,dV (w)\leq \delta ^s$$
and
$$(4)\ \ \ \ \ \ \ \ \ \ \ \ \ \ \ \ \ \overline v_{\delta}(z)=\hat v_\delta (z)\text{ on }\mathbb B_{r_1}, \ \ \ \ \ \ \overline v_{\delta}(z)=v(z)\text{ on }\mathbb B_{2}\backslash\mathbb B_{r_2}.$$
Fix now any $z\in \mathbb B_2\backslash\overline{\mathbb B}_{r_0}$. Since $v=h$ there, we have for any $\delta<\delta_0$,
$$
\aligned\frac {\partial ^2\overline v_\delta}{\partial z_j\partial\overline z_k}(z)&=\int_{\mathbb B_1} [\frac {\partial ^2 h}{\partial z_j\partial\overline z_k} \big(z-\delta\phi (z) w\big) + \delta O(1)] \rho (w)\,dV (w)\\
&=\int_{\mathbb B_1} [\delta_{jk} + \delta O(1)] \rho (w)\,dV (w)\\
&=\delta_{jk} + \delta O(1).
\endaligned
$$

Therefore $\overline v_\delta\in \PSH (\mathbb B_2\backslash\overline{\mathbb B}_{r_0})$, $\forall\delta <\delta_0$, hence $\overline v_{\delta}$ is actually plurisubharmonic in all of $\mathbb B_2$
(if $\delta$ is small enough), as follows from (4). Set
$$
T:=\sum\limits_{j=0}^{n-1} (dd^cv)^j\wedge (dd^c\overline v_{\delta^\varepsilon})^{n-1-j}.
$$
>From (3), (4) and Stokes fomula we get
$$
\aligned\int_K [\check{u}_{\delta }- u] (dd^cv)^n&\leq\int_{\mathbb B_2} [\check{u}_{\delta }- u] (dd^cv)^n\\
&=\int_{\mathbb B_2} [\check{u}_{\delta }- u] [(dd^cv)^n-( dd^c\bar v_{\delta^{\varepsilon}} )^n]+\int_{\mathbb B_2} [\check{u}_{\delta }- u] ( dd^c\bar v_{ \delta^{\varepsilon} } )^n\\
&\leq \int_{\mathbb B_2} [\check{u}_{\delta }- u]  dd^c(v-\bar v_{\delta^{\varepsilon}}) \wedge T+\frac C { \delta^{2n\varepsilon} }\int_{\mathbb B_2} [\check{u}_{\delta }- u]\,dV\\
&\leq \int_{\mathbb B_2} [\bar v_{\delta^{\varepsilon}}-v] dd^c(u - \check{u}_{\delta } ) \wedge T+\frac {C\int_{\mathbb B_2} \Delta u \delta^2} { \delta^{2n\varepsilon} }\\
&\leq \int_{\mathbb B_2} [\bar v_{\delta^{\varepsilon}}-v] dd^c u \wedge T+C\int_{\mathbb B_2} \Delta u \delta^{2(1-n\varepsilon)}\\
&\leq \delta^{\varepsilon s}\int_{\mathbb B_{r_2}} dd^c u \wedge T+C \delta^{2(1-n\varepsilon)}\int_{\mathbb B_2} \Delta u\\
&\leq C[ \delta^{\varepsilon s}\Vert v\Vert_{L^\infty (\Omega )}^{n-1}\int_{\mathbb B_2} \Delta u + \delta^{2(1-n\varepsilon)}\int_{\mathbb B_2} \Delta u]\\
&\leq C \int_{\mathbb B_2} \Delta u \ \delta^{\alpha},\endaligned$$
where $\varepsilon =\frac 2 {s+2n}$, $\alpha =\frac {2s} {s+2n}$.
\end{proof}

\subsection{Proof of Theorem B}\label{thmB}
Below we derive several simple consequences of this characterization. First, the range
of the complex Monge-Amp\`ere operator has the ``$L^p$-property'':

\begin{corollary}
Let $\psi\in\PSH(X,\omega )$ be a H\"older continuous function. Consider a density
 $0 \leq f\in L^p( (\omega+dd^c \psi)^n )$ with $p>1$ and $\int _X f (\omega+dd^c \psi)^n =\int_X\omega^n$.
Then there exists a H\"older continuous $\omega$-plurisubharmonic function $\varphi$ such that
$$
(\omega+dd^c \varphi)^n = f (\omega+dd^c \psi)^n.
$$
In particular $\MAH(X,\omega)$ is a convex set.
\end{corollary}

\begin{proof}
  By H\"older inequality we have
$$\int\limits_K f \omega_\psi^n \leq \Vert f\Vert_{L^p( \omega_\psi^n )}[\omega_\psi^n (K)]^{1-\frac 1 p},$$
for any Borel subset $K$ of $X$. This implies that $f \omega_\psi^n\in\mathcal H(\infty )$. On the other hand, by H\"older inequality we have
$$
\int\limits_K [\check{u}_{\delta }- u] f \omega_\psi^n \leq \Vert f\Vert_{L^p( \omega_\psi^n )}[\int\limits_K [\check{u}_{\delta }- u] \omega_\psi^n]^{1-\frac 1 p}\leq C\delta^\alpha,
$$
for all $u\in{\PSH}\cap L^\infty (\Omega)$, $K\subset\subset D\subset\subset \Omega$ and local chart $\Omega$. Therefore using Theorem 2.1 there exists a H\"older continuous $\omega$-psh function $\varphi$ such that $\omega_\varphi^n = f \omega_\psi^n$.

Fix $\mu_1=\MA(\phi_1), \mu_2=\MA(\phi_2) \in \MAH(X,\omega)$ and set $\mu=(\mu_1+\mu_2)/2$.
Observe that $\psi:=(\phi_1+\phi_2)/2 \in \PSH(X,\omega) \cap\Holder(X,\omega)$ satisfies
$$
(\omega+dd^c \psi)^n \geq \frac{1}{2^n} (\mu_1+\mu_2)
$$
hence $\mu=f (\omega+dd^c \psi)^n$ with bounded density $0 \leq f \leq 2^{n-1}$.
It therefore follows from the first part of the corollary that $\mu$ also belongs to $\MAH(X,\omega)$, hence
the latter is convex.
\end{proof}

We also note that the range of the complex Monge-Amp\`ere operator has the product property.

\begin{corollary}
Let $(X_1,\omega_1), (X_2,\omega_2)$ be two compact K\"ahler manifolds
of dimension $n_1,n_2$,
normalized so that $\int_{X_1} \omega_1^{n_1}=\int_{X_2} \omega_2^{n_2}=1$.
Fix $\mu_1,\mu_2$ two probability measures on $X_1,X_2$.
The  following are equivalent:
\begin{itemize}
\item[\rm i)] $\mu_1\in \MAH(X_1,\omega_1)$ and $\mu_2\in \MAH(X_2,\omega_2)$.
\smallskip
\item[\rm ii)] $\mu=\mu_1\times\mu_2 \in \MAH (X_1 \times X_2, \omega)$, where
$$
\omega=\left( \begin{array}{c} n_1+n_2 \\ n_1 \end{array} \right)^{-1/(n_1+n_2)} [\omega_1+\omega_2].
$$
\end{itemize}
\end{corollary}

Here $\mu=\mu_1 \times \mu_2$ denotes the product (probability) measure on $X_1 \times X_2$,
and we still denote by $\omega_1,\omega_2$ the semi-positive forms on $X_1 \times X_2$
obtained by pulling-back $\omega_1,\omega_2$ on each factor.

\begin{proof}
i) $\Rightarrow$ ii) Assume that $\mu_1=(\omega_1+dd^c u_1 )^{n_1}$ and $\mu_2=(\omega_2+dd^c u_2 )^{n_2}$
where $u_1$, $u_2$ are H\"older continuous $\omega_i$-psh functions on $X_1$, $X_2$. Pulling back these
forms and functions on $X=X_1 \times X_2$ and observing that
$(\omega_i+dd^c u_i)^{1+n_i} \equiv 0$, one obtains
$$
\mu=\mu_1\times\mu_2=(\omega+dd^c u )^{n_1+n_2}
\text{ with }
u=\frac{[u_1+u_2]}{\left( \begin{array}{c} n_1+n_2 \\ n_1 \end{array} \right)^{1/(n_1+n_2)}}
$$
so that $\mu \in \MAH(X,\omega)$.

ii) $\Rightarrow$ i) Since $\mu$ satisfies iii) in Theorem \ref{thm:char}
 we infer that $\mu_1,\mu_2$ satisfy the same property. Using Theorem \ref{thm:char} again
 thus yields $\mu_1\in \MAH(X_1,\omega_1)$, $\mu_2\in \MAH (X_2,\omega_2)$.
\end{proof}

\section{Measures with symmetries}\label{symmetries}
Generalizing Skoda's celebrated result \cite{Sk}, Dinh-Nguyen-Sibony have observed recently \cite{DNS}
 that if $\mu$ is the Monge-Amp\`ere measure of a H\"older-continuous quasi-psh function,
 then
 $$
 \exp(-\varepsilon \PSH(X,\omega)) \subset L^1(\mu)
 $$
 for $\varepsilon>0$ small enough.
 We show here that the converse holds when $\mu$ moreover has radial or toric singularities.  The general case is open, see
 however \cite{Hi} for some
partial results.

\subsection{Exponential integrability, Lelong numbers and symmetries- basic results}

Note for later use that if $ \exp(-\varepsilon \PSH(X,\omega)) \subset L^1(\mu)$, then for all $x \in X$ and $0<r<<1$,
$$
\mu(\mathbb B(x,r)) \leq C r^{\varepsilon}
$$
and  $\mu(K) \leq C T(K)^{\varepsilon}$ for all Borel sets $K$, where $T$ denote the Alexander-Taylor capacity (see \cite{GZ1}).
This implies that for all $A>1$, there exists $C_A>0$ such that
$$
\mu(K) \leq C_A \Capa_{\omega}(K)^A, \text{ for all Borel set } K,
$$
where $\Capa_{\omega}$ denotes the Monge-Amp\`ere capacity. In other words, $\mu$ is very well dominated by the Monge-Amp\`ere capacity (it satisfies the condition ${\mathcal H}(\infty)$).

\medskip

 Let  $u$ be a psh function defined near the origin in $\mathbb C^n$, with a
{\it radial singularity}, i.e. such that $u(z)=u(\Vert z\Vert )$ for all $z$. It is then standard that $u$ can be written as
$u(z)=\chi \circ L(z)$ where $L(z)=\log \Vert z\Vert $ and $\chi$ is a convex increasing function defined
in a neighborhood of $-\infty$. Note that

-- the function $u$ is bounded if and only $\chi(-\infty)>-\infty$;

-- the  Lelong number $\nu(u,0)$ is non zero if and only if $\chi(t) \sim \nu(u,0) t$ near $-\infty$,

\noindent which is the maximal growth that $\chi$ can have at $-\infty$. Alternatively, $\nu(u,0)=0$ if and only if $\chi'(-\infty)=0$.
The following elementary computation is left to the reader:

\begin{lemma} \label{lem:calcul}
Let $u=\chi \circ L$ be a radial plurisubharmonic function defined in a ball $\mathbb B \ni 0$. Assume that $\chi$ is $\mathcal C^2$ smooth.
Then $u$ belongs to the domain of definition of the Monge-Amp\`ere operator and
$$
(dd^c u)^n=\nu(u,0)^n \delta_0+ c_n (\chi' \circ L)^{n-1} \chi'' \circ L \frac{dV}{\Vert z\Vert^{2n}}.
$$
\end{lemma}

Here $\delta_0$ denotes the Dirac mass at the origin. Note in particular that when $\nu(u,0)=0$ then
the Monge-Amp\`ere measure $(dd^c u)^n$ is absolutely continuous with respect to Lebesgue measure.

\smallskip

 A similar formula can be derived for Monge-Amp\`ere measures with {\it toric symmetries}, but
 we will not use it: we will handle the toric case by using Theorem \ref{thm:char}, whereas the radial
 case will be treated directly, using Lemma \ref{lem:calcul} (the direct method yields better exponents).

\subsection{The radial case}

We obtain here a complete  description of those radial measures  which belong to $\MAH(X,\omega)$.

\begin{proposition} \label{pro:radialholder}
Let $\mu$ be  a probability measure on $X$ which is smooth but at finitely many points where it has
a radial singularity.  The following are equivalent:

{\rm i)}  $\exp (-\varepsilon \PSH(X,\omega)) \subset L^1(\mu)$ for all $0 < \varepsilon < \varepsilon_0$;

{\rm ii)} $\Vert z-a\Vert^{-\varepsilon} \in L^1(\mu)$ for all $0 < \varepsilon < \varepsilon_0$ and $a \in X$;

{\rm iii)} $\mu(\mathbb B(a,r)) \leq C r^{\varepsilon}$ for all $0 < \varepsilon < \varepsilon_0$ and $a \in X$;

{\rm iv)}  $\mu=(\omega+dd^c \phi)^n$, where $\phi \in \PSH(X,\omega)$ is H\"older continuous with exponent
$\alpha$ arbitrarily close to $\varepsilon_0/n$.
\end{proposition}

\begin{proof}
The implication i)~$\Rightarrow$~ii) is obvious. The equivalence
ii)~$\Leftrightarrow$~iii) is immediate. The implication
iv)~$\Rightarrow$~iii) is classical (successive integration by
parts against a cut-off function with support in a corona of radii
$jr, (j+1)r$) and holds for general (non radial) measures. The
implication iv)~$\Rightarrow$~i) was  obtained in
\cite{DNS}, also for general measures.
 In the sequel we thus focus on the remaining implication
ii)~$\Rightarrow$~iv).

Let $a \in X$ be one of the finitely many singular points. We fix  a local chart near $a$ such that
$a=0$ is the origin and locally $\mu=(dd^c u)^n$ with $u=\chi \circ L$, $L(z)=\log \Vert z\Vert $ and
$\chi$ convex increasing. Observe that  $u$ is bounded and $\chi'(-\infty)=0$.
By Theorem \ref{thm:char} it is enough to check that $u$ is H\"older continuous at point $a$, which is
equivalent to showing that
$$
0 \leq \chi(t)-\chi(-\infty) \leq C \exp( \delta t) \text{ as } t \rightarrow -\infty,
$$
for some positive constants $C,\delta>0$.

By assumption there exists $\varepsilon>0$ such that $\Vert z\Vert^{-\varepsilon} \in L^1(\mu)$. We infer from Lemma \ref{lem:calcul} that
\begin{eqnarray*}
\lefteqn{
\int_0 \frac{1}{\Vert z\Vert^{\varepsilon}} d \mu=c \int_0 (\chi' \circ L)^{n-1} \chi'' \circ L \frac{dV(z)}{\Vert z\Vert^{2n+\varepsilon}}} \\
&=&c'\int_{-\infty} (\chi'(t))^{n-1} \chi''(t) e^{-\varepsilon t} dt  <+\infty.
\end{eqnarray*}

We now integrate by parts, in finite time, to obtain
$$
\varepsilon \int_{-A} (\chi')^n \exp(-\varepsilon t)
dt=(\chi')^n(-A)\exp(+\varepsilon A) +O(1)  . 
$$

We claim that $\int_{-\infty} (\chi')^n \exp(-\varepsilon t) dt$ is finite.
So is the limsup on the right hand side, hence
$\chi'(t) \leq C \exp(\varepsilon t/n)$, which yields
$$
\chi(t)-\chi(-\infty) \leq C' \exp(\varepsilon t/n).
$$
Therefore $u(z)-u(0) \leq C' \Vert z\Vert^{\varepsilon/n}$, i.e. $u$ is H\"older continuous.

\smallskip

It remains to prove the claim. If $\int_{-\infty} (\chi')^n \exp(-\varepsilon t) dt=+\infty$,
then \newline $(\chi')^n(-A)\exp(+\varepsilon A) \rightarrow +\infty$ as $A \rightarrow +\infty$.
Set
$$
h(t)=(\chi')^n(t) \exp(-\varepsilon t) \; \; \text{ and } \; \;  H(x)=\int_x^0 h(t)dt.
$$
Thus $H(x) \rightarrow +\infty$ as $x \rightarrow -\infty$ and
$$
\varepsilon H(x)=H'(x)+O(1)  . 
$$
We let the reader check that this implies
$H(x)=\lambda \exp(-\varepsilon x)+O(1)$ for some constant $\lambda \geq 0$.
Now $\chi'(t) \rightarrow 0$ as $t \rightarrow -\infty$ so $h(t)=o(\exp(-\varepsilon t))$
and $H(t)=o(\exp(-\varepsilon t))$. This forces $\lambda=0$, hence $H(t)=O(1)$.
\end{proof}

\subsection{The toric case}

We now consider the case of probability measures $\mu$ which are smooth but at finitely many points where they have ``toric singularities'' the origin $0 \in \mathbb C^n$ is called a toric singularity for the measure $\mu=(dd^c u)^n$, $u$ psh and bounded, if $u$ is $(S^1)^n$-invariant, i.e.
$$
u(z_1,\ldots,z_n)=u(|z_1|, \ldots,|z_n|),
\; \;
\forall z=(z_1,\ldots,z_n) \in \Delta^n.
$$
We will call these measures  {\it toric measures} for short.

\begin{proposition}
{\sl Let $\mu$ be a toric measure in the unit polydisk $\Delta^n \subset \mathbb C^n$.
Assume that for  all $0<r<\frac 1 2$ and $j=1,...,n$,
$$
\mu (\Delta\times ...\times\Delta_j (r)\times ..\times\Delta)\leq C r^\alpha, \; \text{ where } C,\alpha >0.
$$
Then
$$
\int_{\Delta_n (t)} [\check{u}_{\delta }(z)- u(z)] d\mu \leq C(t) \delta^{\beta},
$$
for all $0<t<1$ and $u\in \PSH\cap L^\infty (\Delta^n)$ with $0\leq u\leq 1$.}
\end{proposition}

\begin{proof}
Set
$
Tu(z)=\frac 1 { ( 2\pi )^n } \int_{[0,2\pi]^n}u(e^{i\theta_1}|z_1|,...,e^{i\theta_n}|z_n|) d\theta_1 ... d\theta_n.
$
Note that \newline $Tu(z)=Tu(|z_1|,...,|z_n|)$ is increasing and logarithmically convex. This implies that
$$
 Tu(|z_1|+\delta_1,...,|z_n|+\delta_n)-Tu(|z_1|,...,|z_n|)|\leq C\sum\limits_{j=1}^n \log \left( 1+\frac {\delta_j} {|z_j|}
 \right),
$$
for all $z\in\Delta_n (1/2)$.
It follows from  Fubini theorem that
$$
\aligned  T \check{u}_{\delta }(z)&=\frac 1 { ( 2\pi )^n } \int\limits_{[0,2\pi]^n} \check{u}_{\delta } (e^{i\theta_1}|z_1|,...,e^{i\theta_n}|z_n|) d\theta_1 ... d\theta_n\\
&=\frac 1 { ( 2\pi )^n } \int\limits_{[0,2\pi]^n} \frac 1 { c_n\delta^n }\int\limits_{\mathbb B_\delta} u(e^{i\theta_1}|z_1|+w_1,...,e^{i\theta_n}|z_n|+w_n)\,dV (w) d\theta_1 ... d\theta_n\\
&=\frac 1 { ( 2\pi )^n } \int\limits_{[0,2\pi]^n} \frac 1 { c_n\delta^n }\int\limits_{\mathbb B_\delta(|z_1|,...,|z_n|)} u(e^{i\theta_1}\xi_1,...,e^{i\theta_n}\xi_n)\,dV (\xi) d\theta_1 ... d\theta_n\\
&=\frac 1 { c_n\delta^n }\int\limits_{ \mathbb B_\delta(|z_1|,...,|z_n|) } Tu( \xi )\,dV (\xi)\\
&\leq Tu(|z_1|+\delta,...,|z_n|+\delta).
\endaligned$$

Since $\mu$ is toric,
$$
 \int_{\Delta_n (1/2)} [\check{u}_{\delta }(z)- u(z)] d\mu =
\int_{\Delta_n (1/2)} [T\check{u}_{\delta }(z)- Tu(z)] d\mu ,
$$
thus
$$
\aligned \int_{\Delta_n (1/2)} [\check{u}_{\delta }(z)- u(z)] d\mu &\leq C \sum\limits_{j=1}^n \int_{\Delta_n (1/2)} \log \left( 1+\frac {\delta} {|z_j|} \right) d\mu\\
&\leq n C \int_0^{1/2} \frac {\delta t^{\alpha}}{t^2+\delta t}  dt\\
 &\leq C' \delta^{\beta}
\endaligned
$$
with $\beta=\alpha/(\alpha+2)$, as can be checked by cutting the integral into two pieces $\int_0^{\delta^{\gamma}}+\int_{\delta^\gamma}^{1/2}$,
where $\gamma=1/(\alpha+2)$.
\end{proof}

\begin{corollary}
A toric probability measure $\mu$ belongs to $\MAH(X,\omega)$ if and only
if\break
$\exp(-\varepsilon \PSH(X,\omega)) \subset L^1(\mu)$ for some $\varepsilon>0$.
\end{corollary}

\begin{proof}
If $\mu$ belongs to $\MAH(X,\omega)$, then $\exp(-\varepsilon \PSH(X,\omega)) \subset L^1(\mu)$ for some $\varepsilon>0$,
as follows from \cite{DNS}.
Assume now that $\exp(-\varepsilon \PSH(X,\omega)) \subset L^1(\mu)$ for some $\varepsilon>0$.
As explained earlier, this implies that $\mu$ is very well dominated by the Monge-Amp\`ere capacity, in particular
$\mu \in {\mathcal H}(\infty)$. The previous proposition shows that item (iii) of Theorem \ref{thm:char}
applies, hence $\mu \in \MAH(X,\omega)$.
\end{proof}

In view of the above proofs, one may wonder whether all probability measures satisfying condition ${\mathcal H}(\infty)$
belong to $\MAH(X,\omega)$. The following example shows this is far from being the case.

\begin{example}
We assume here $(X,\omega)=(\mathbb P^1,\omega_{FS})$ is the Riemann sphere equipped with the Fubini-Study form.
We let $\phi \in \PSH(X,\omega)$ be a function that is smooth in $\mathbb P^1$ but at one point which we choose as the origin
$0$ in some affine chart $\mathbb C$ and so that
$$
\phi(z)=\exp\left(-\sqrt{-\log |z|} \right)-\frac{1}{2} \log [1+|z|^2]
$$
near the origin. The reader will easily check, following the arguments in Example 4.2 in \cite{BGZ}, that
$\mu=\omega+dd^c \phi$ is very well dominated by the logarithmic capacity , in particular satisfies ${\mathcal H}(\infty)$,
although $\phi$ is not H\"older continuous.
\end{example}

\section{The case of big cohomology classes}\label{big}

\begin{proof}[Proof of Theorem D]
In order to deal with the general case of big cohomology classes, we use again  the regularization techniques of the first author, coupled now with Proposition \ref{proEGZ}.

\smallskip

We let $\f$ be a $\theta$-psh function solution of $(\theta+dd^c \f)^n=\mu$,
where the density  $f \geq 0$ of $\mu$ with respect to a smooth volume form belongs to $L^p$ for some $p>1$. The solution is unique
up to an additive constant, it is $\theta$-psh with minimal singularities (see \cite{BEGZ}). We
can thus assume, without loss of generality, that $-C_0 + V_{\theta}  \leq \f \leq  V_{\theta}$. We let
$$
\f \mapsto \rho_\d \f
$$
again denote the regularization operator defined  in (2.1). As in the K\"ahler case
$t \mapsto \rho_t \f+K t^2$ is increasing for $0 \leq t \leq \d_0$ and some constant $K$.

We consider the Kiselman-Legendre transform,
$$
\p_{c,\d}(z):=\inf_{t \in{}]0,\d]} \left\{ \rho_t \f(z)+K t^2 -c \log (t/\d) \right\},
$$
where $0 \leq \d \leq \d_0$ and $c>0$ will be carefully chosen below. Observe that
$$
\f \leq \p_{c,\d} \leq \rho_\d \f + K\d^2.
$$
The fundamental curvature estimate is now
$$
\theta + dd^c \p_{c,\d} \geq -(Ac+K\d^2) \omega
$$
for some constant $A>0$.
Since the coholomogy class $\Theta=\{\theta\}$ is big, there exists a $\theta$-psh function $\p_0$ on $X$ such
 that $\theta + dd^c \p_0 \geq \e_0 \omega$, for some small constant $\e_0 > 0$. Subtracting a large constant, we can always assume that $\p_0 \leq 0$ hence $\p_0 \leq V_{\theta}$.

It follows that the function
$$
\f_{c,\d} := \frac{A c + K \d^ 2}{\e_0} \p_0 + \left(1 - \frac{A c
+ K \d^ 2}{\e_0}\right) \p_{c,\d} 
$$
 is $\theta$-plurisubharmonic on $X$. Fix $0 < \d < \d_0$ and choose $c > 0$ such that
$$
Ac+K\d^2= \e_0  \d^{\alpha}, \ \ \text{where} \ \alpha := 2\gamma,
$$
and observe that $c= \e_0 A^{-1}\d^{\alpha} - K A^{-1} \d^2 =O(\d^{\alpha})$.
In the sequel we set
$$
\f_\d := \f_{c,\d}.
$$
Since $\p_0 \leq V_{\theta} \leq \f + C_0$, we see from the definition that on the ample locus,
\begin{eqnarray*}
\f_\d - \f & = & \d^{\alpha} (\p_0 - \f) + (1 - \d^{\alpha}) (\p_{c,\d} - \f) \\
&\leq & C_0\d^\alpha + (1 - \d^{\alpha}) (\rho_{\d} \f - \f + K \d^2).
\end{eqnarray*}
Furthermore, since $\f\leq V_\theta\leq 0$, we get $\varrho_\d\f\leq 0$, thus $\psi_{c,\d} \leq K\d^2\leq C_0\d^\alpha$ if $\d\le\d_0$ small enough, and so $\varphi_{\d}\le C_0\d^\alpha$. This implies $\psi:=\f_\d-C_0\d^\alpha \leq V_\theta$. By Proposition \ref{proEGZ}, it follows that
\begin{eqnarray*}
\sup_X (\f_\d-\f) & \leq & B_0 \Vert  \max(\f_\d-\f-C_0 \d^\alpha,0)\Vert_{L^1(X)}^\gamma + C_0\d^\alpha \\
& \leq & B_0 \Vert  \rho_\d \f+K\d^2-\f\Vert_{L^1(X)}^\gamma + C_0\d^\alpha
\end{eqnarray*}
for some constant $B_0>0$ which depends only on $\gamma$ and the uniform norm of $\f - V_{\theta}$.

Applying Lemma 2.3,  the last estimate yields
$$
\sup_X (\f_\d-\f) \leq C_1 \d^{\alpha},
$$
where $C_1 := B_0 C_\omega + K^\gamma + C_0$ and $C_\omega$ is the constant in Lemma 2.3.

This inequality $\f_\d \leq \f+C_1 \d^{\alpha}$  yields a uniform lower bound on the parameter
$t=t(z)$ which realizes the infimum in the definition of $\f_\d(z)$ for a fixed $z \in \Omega$. Namely the last inequality gives
\begin{eqnarray*}
\f_{\d} (z) - \f (z)& = &\d^{\alpha} (\p_0 (z) {-} \f (z)) + \left(1 {-} \d^{\alpha}\right) (\rho_t \f (z){+} K t^2 {-} \f (z) {-} c \log (t\slash \d)) \\
&\leq& C_1 \d^{\alpha}.
\end{eqnarray*}
Since  $ V_{\theta} - \f \geq 0$ and $\rho_t \f (z) + K t^2 - \f (z) \geq 0$, it follows that
$$
c (1 - \d^{\alpha}) \log [ t(z)/\d] \geq \d^{\alpha} (\p_0 (z) - V_{\theta} (z) - C_1).
$$
Since $c= \e_0 A^{-1} \d^{\alpha} - K A^{-1} \d^2$, the choice $\d \leq \d_1 := \min \{\d_0, (\e_0\slash 2 K)^{1 \slash (2-\alpha)}\}$ yields
$c\geq \frac{1}{2}\e_0 A^{-1} \d^{\alpha}$ and therefore
$$
t(z) \geq  \d \kappa (z),
$$
where
\begin{equation}
\label{eq:kappa}
 \kappa (z) :=  \exp\left( C_2 (\p_0 (z) - V_{\theta} (z) - C_1\right),
\end{equation}
\begin{equation} \label{eq:C2}
 C_2 := \frac{2 A}{\e_0 (1-\d_0^{\alpha})}.
\end{equation}

We are now in position to conclude.
Fix $z \in Amp(\Theta)$. Since $ t(z) \geq \kappa (z) \d$ and $t\mapsto
\rho_t\f+Kt^2$ is increasing,  we get
\begin{eqnarray*}
\rho_{\kappa (z) \d} \f(z) -\f(z) & \leq & \rho_{t(z)}  \f(z)+K t(z)^2-\f(z) \\
& = & \p_{c,\d} (z) - \f (z) = \frac{1}{1-\d^\alpha} (\f_\d (z) - \d^\alpha\psi_0(z)),
\end{eqnarray*}
and by the above and the assumption $\varphi\leq V_\theta\leq 0$ we find
\begin{eqnarray*}
\f_\d - \d^\alpha\psi_0 & \leq & \f + C_1\d^\alpha - \d^\alpha\psi_0 \leq
 C_1\d^\alpha + \d^\alpha (V_\theta - \psi_0),\\
\rho_{\kappa (z) \d} \f(z) -\f(z)
& \leq & (1- \d_0^{\alpha})^{-1})  \d^{\alpha} (C_1 + V_\theta(z) - \p_0(z)).
\end{eqnarray*}

Replacing $\d$ by $\kappa (z) ^{-1}\d$ and using (\ref{eq:kappa}), we obtain for $\d \leq \d_0 \kappa (z)$,
\begin{eqnarray}
\label{eq:Halpha}
\rho_{ \d} \f(z) -\f(z)& \leq &(1- \d_0^{\alpha})^{-1}\d^{\alpha}
(C_1 + V_\theta(z) - \p_0(z)) \cdot
\exp \left(\alpha C_2 (C_1 + V_\theta(z) - \p_0(z))\right)\nonumber  \\
&\leq & C_3 \exp \left(2\alpha C_2 (C_1 + V_\theta(z) -
\p_0(z))\right), 
\end{eqnarray}
where
\begin{equation}
\label{eq:C3}
C_3 := (\alpha C_2)^{-1} (1 - \d_0^{\alpha})^{-1}.
\end{equation}
 This finishes the proof of Theorem D, since $\p_0 (z) - V_{\theta} (z)$ is locally bounded from below on $Amp({\Theta})$ as well as $\kappa(z)$ given by (\ref{eq:kappa}).
 \end{proof}

\section*{Appendix}
We briefly explain below how bounds on the curvature may be used to control the differential of the exponential mapping. This is essentially a variation on the theme of Jacobi vector fields.

\subsection*{Estimates for the differential of the exponential}
For accurate computations with the exponential we need to control its differential in terms of the curvature. To this end we determine the Jacobi equations which  calculate the variation of geodesics.

Let namely $u\rightarrow u+v$ be a small perturbation of the geodesic $t\rightarrow u(t)$ with initial velocity $\zeta$. Its linearization satisfies
\begin{equation}\label{linearization}
 \frac{d^2v_m}{dt^2}=\sum_{j,k,l}R_{j\bar{k}l\bar{m}}\bar{v}_k\frac{du_j}{dt}\frac{du_l}{dt}+O(|u(t)|).
\end{equation}
Moreover if $\connection$ denotes the Levi-Civita connection with respect to $\omega$ then along the geodesic $u(t)$ one can compute
\begin{equation}\label{nabla}
\Big(\frac{\connection\zeta}{dt}\Big)_m=\frac{d\zeta_m}{dt}-\sum_{j,k,l}R_{j\bar{k}l\bar{m}}\frac{du_j}{dt}\zeta_l+O(|u(t)|^2)\zeta
\end{equation}
\begin{equation}\label{2nabla}
\Big(\frac{\connection^2\zeta(t)}{dt^2}\Big)_m=\frac{d^2\zeta_m}{dt^2}-\sum_{j,k,l}R_{j\bar{k}l\bar{m}}\frac{d\bar{u}_k}{dt}\frac{du_j}{dt}
\zeta_l+O(|u(t)|)\zeta.
\end{equation}
Let us now put $\zeta=v$. Then the Jacobi equation takes the intrinsic form
\begin{equation*}
\Big(\frac{\connection^2v(t)}{dt^2}\Big)_m=\sum_{j,k,l}R_{j\bar{k}l\bar{m}}\bar{v}_k\frac{du_j}{dt}\frac{du_l}{dt}
-\sum_{j,k,l}R_{j\bar{k}l\bar{m}}\frac{d\bar{u}_k}{dt}\frac{du_j}{dt}v_l.
\end{equation*}

In particular the formula holds at $\zeta:=u'(0)$.
Thus if the curvature is bounded by the constant $R_0^2$ (the square being taken for the ease of notation), then
$$|(\connection_{v(t)}\connection_{v(t)})|\leq2R_0^2|\zeta|^2|v|.$$
This is a vector analogue of the scalar equation $y''=2py$. By Gronwall's lemma the solution to the corresponding Cauchy problem with data $v(0)=v_0,\ Dv(0)=v_1$ is estimated by
$$|v(t)|\leq |v_0|\cosh(\sqrt{2}R_0|\zeta|t)+\frac{|v_1|}{\sqrt{2}R_0|\zeta|}\sinh(\sqrt{2}R_0|\zeta|t).$$

Let us denote by $\tau_{z,\zeta}(t):T_ZX\rightarrow T_{\exp_z(t\zeta)}X$ the parallel translation along the geodesic. Let also $\tilde{v}(t):=\tau_{z,\zeta}(t)^{-1}v(t)\in T_zX$. Then $\tilde{v}$ satisfies the analogous equation with curvature transported back to $T_zX$. Thus
\begin{equation}\label{taueq}
 |\tilde{v}(t)-v_0-v_1t|\leq
 |v_0|\cosh(\sqrt{2}R_0|\zeta|t)+\frac{|v_1|}{\sqrt{2}R_0|\zeta|}\sinh(\sqrt{2}R_0|\zeta|t)-|v_0|-|v_1|t.
\end{equation}

The differential of the ordinary exponential mapping evaluated at $(h,\eta)\in T(TX)_{(z,\xi)}\simeq T_zX\otimes T_zX$ is precisely $v(1)$ for the solution of the Cauchy problem $v(0)=h, \connection v(0)=\eta$. Thus (\ref{taueq}) gives us the bound

$$|\tau_{z,\zeta}(1)^{-1}d\exp_z(\zeta)(h,\eta)-(h+\eta)|\leq  h\cosh(\sqrt{2}R_0|\zeta|)+\frac{\eta}{\sqrt{2}R_0|\zeta|}\sinh(\sqrt{2}R_0|\zeta|)-h-\eta.$$

If $|\zeta|$ is small ($|\zeta|\leq\frac{\varepsilon}{2R_0}$, say), then elementary Taylor expansion gives us the bound

$$|\tau_{z,\zeta}(1)^{-1}d\exp_z(\zeta)(h,\eta)-(h+\eta)|\leq (1+O(\varepsilon))(c_1\varepsilon^2|h|+c_2\varepsilon^2|\eta|).$$

Thus there exists some uniform $\varepsilon_0$ such that in the
balls $|\zeta|\leq\frac{\varepsilon}{\sqrt{2}R_0}$ for any
$\varepsilon\leq\varepsilon_0$ the differential is a
diffeomorphism and is even $O(\varepsilon^2)$ close to the
identity.

\begin{Remark}
 Similar estimates can be obtained in the Hermitian case either, geodesics being defined by the Chern connection rather than the
 Levi-Civita connection. One then has to assume additionally a uniform bound on $|\partial\omega|_{\omega}$ and
 $|\connection(\partial\omega)|_{\omega}$ to accommodate the presence of torsion. However, replacing $\exp$ by $\exph$ as was done in \cite{De94} and \cite{BD}
would be a challenge, because we would then need an ``effective'' version
of E.~Borel's theorem to show that $\exph$ can be chosen to satisfy the same
estimates as~$\exp$, and this is certainly non trivial.
\end{Remark}

\noindent {\bf Acknowledgements}. This project was initiated during the BIRS
conference ''Complex Monge-Amp\`ere equation'' held in
2009. The authors would like to acknowledge the perfect working
conditions provided by the organizers.   Part of the
research was done while the second and fifth authors were visiting
the Erwin Schr\"odinger institute in Vienna in 2009. They would
like to express their gratitude for the hospitality. The second
author was supported by Polish ministerial grant ``Iuventus Plus''
and Kuratowski fellowship granted by the Polish Mathematical
Society (PTM) and Polish Academy of Science (PAN). The research
was done while the fourth author was a postdoctoral research
fellow at Centro Internazionale per la Ricerca Matematica, Trento,
Italy. He would like to thank the members of the Institute for
their kind hospitality.  The second and fifth authors were
partially supported by NCN grant 2011/01/B/ST1/00879.

\bigskip

\noindent {Jean-Pierre Demailly: Institut Universitaire de France et Universit\'e Grenoble I, 100 rue des Maths
38402 Saint-Martin d'H\`eres,  France; 

\noindent e-mail: {\tt demailly@fourier.ujf-grenoble.fr}} \\ \\
{S\l awomir Dinew: Rutgers University, Newark, NJ 07102, USA;\\  Jagiellonian University 30-348 Krakow, Lojasiewicza 6, Poland;\\ e-mail: {\tt slawomir.dinew@im.uj.edu.pl}}\\ \\
{Vincent Guedj: Institut Universitaire de France et Institut de Math\'ematiques de Toulouse, Universit\'e Paul Sabatier, 31602 Toulouse Cedex 09, France; \\
e-mail: {\tt vincent.guedj@math.univ-toulouse.fr}}\\ \\
{Pham Hoang Hiep: Hanoi National University of Education, Tuliem-Hanoi-Vietnam; \\
e-mail: {\tt phhiep\_vn@yahoo.com}}\\ \\
{S\l awomir Ko\l odziej:  Jagiellonian University 30-348 Krakow, Lojasiewicza 6, Poland; \\
e-mail: {\tt slawomir.kolodziej@im.uj.edu.pl}}\\ \\
{Ahmed Zeriahi:  Institut de Math\'ematiques de Toulouse,  Universit\'e Paul Sabatier, 31602 Toulouse Cedex 09, France; 
\\ e-mail: {\tt ahmed.zeriahi@math.univ-toulouse.fr}}

\end{document}